\documentclass{amsart}
\usepackage{enumerate}
\usepackage{ifpdf}
\usepackage{amsmath}
\usepackage{amsthm}
\usepackage{amscd}
\usepackage{amssymb}
\usepackage{amsfonts}
\usepackage{amssymb}
\usepackage{amsthm} 
\usepackage[hyperfootnotes=false]{hyperref}
\usepackage{setspace}
\usepackage{lipsum} 
\usepackage{pstricks}
\usepackage{tikz}

\usepackage[draft]{todonotes}   
\usepackage{graphicx}

\newtheorem{thm}{Theorem}[section]
\newtheorem*{thm*}{Theorem}

\newtheorem{prop}[thm]{Proposition}

\newtheorem*{cor*}{Corollary}
\newtheorem*{example*}{Example}
\newtheorem{lemma}[thm]{Lemma}

\theoremstyle{definition}
\newtheorem{defn}[thm]{Definition}

\theoremstyle{remark}
\newtheorem{remark}[thm]{Remark}

\newcommand{\ii}{\textrm{i}}
\newcommand{\oo}{\textrm{o}}
\newcommand{\la}{\lambda}

\newcommand{\CC}{\mathbb C}

\newcommand{\Id}{\textrm{Id}}

\newcommand{\Ui}{U^\textrm{i}}
\newcommand{\Uo}{U^{\textrm{o}}}
\newcommand{\Ua}{U^\textrm{a}}
\newcommand{\Ub}{U^{\textrm{b}}}

\newcommand{\Wi}{W^\textrm{i}}
\newcommand{\Wo}{W^{\textrm{o}}}
\newcommand{\Wa}{W^\textrm{a}}
\newcommand{\Wb}{W^{\textrm{b}}}

\newcommand{\UiF}{U^{\textrm{i},F}}
\newcommand{\UoF}{U^{\textrm{o},F}}

\newcommand{\Uif}{U^{\textrm{i},f}}
\newcommand{\Uof}{U^{\textrm{o},f}}

\newcommand{\UiG}{U^{\textrm{i},G}}
\newcommand{\UoG}{U^{\textrm{o},G}}

\newcommand{\Uig}{U^{\textrm{i},g}}
\newcommand{\Uog}{U^{\textrm{o},g}}

\newcommand{\WiG}{W^{\textrm{i},G}}
\newcommand{\WoG}{W^{\textrm{o},G}}

\newcommand{\phii}{\phi^{\textrm{i}}}
\newcommand{\phio}{\phi^{\textrm{o}}}
\newcommand{\phiif}{\phi^{\textrm{i},f}}
\newcommand{\phiof}{\phi^{\textrm{o},f}}
\newcommand{\phiig}{\phi^{\textrm{i},g}}
\newcommand{\phiog}{\phi^{\textrm{o},g}}
\newcommand{\phiiF}{\phi^{\textrm{i},F}}
\newcommand{\phioF}{\phi^{\textrm{o},F}}
\newcommand{\phiiG}{\phi^{\textrm{i},G}}
\newcommand{\phioG}{\phi^{\textrm{o},G}}

\newcommand{\psio}{\psi^{o}}

\newcommand{\PhiiG}{\Phi^{\textrm{i},G}}
\newcommand{\PhioG}{\Phi^{\textrm{o},G}}
\newcommand{\PhiaG}{\Phi^{\textrm{a},G}}

\newcommand{\PhiiH}{\Phi^{\textrm{i},H}}

\newcommand{\Psia}{\Psi^{\textrm{a}}}
\newcommand{\Psib}{\Psi^{\textrm{b}}}

\newcommand{\thetai}{\theta_{\textrm{i}}}
\newcommand{\thetao}{\theta_{\textrm{o}}}
\newcommand{\thetaa}{\theta_{\textrm{a}}}
\newcommand{\thetab}{\theta_{\textrm{b}}}

\newcommand{\Gi}{G_{\textrm{i}}}
\newcommand{\Go}{G_{\textrm{o}}}
\newcommand{\Ga}{G_{\textrm{a}}}
\newcommand{\Gb}{G_{\textrm{b}}}

\newcommand{\PsiaG}{\Psi^{\textrm{a},G}}
\newcommand{\PsibG}{\Psi^{\textrm{b},G}}

\newcommand{\PsibH}{\Psi^{\textrm{b},H}}

\newcommand{\Ve}{V_{\epsilon}}

\date{\today}

\ifpdf
\author{Liz Vivas}
\address{The Ohio State University\\
Columbus, OH\\
USA}
\email{vivas.3@osu.edu}

\title[Local dynamics of parabolic skew-products]{Local dynamics of parabolic skew-products}

\begin{document}
\bibliographystyle{plain}

\begin{abstract} The local dynamics around a fixed point has been extensively studied for germs of one and several complex variables. In one dimension, there exist a complete picture of the trajectory of the orbits on a whole neighborhood of the fixed point. In dimensions larger or equal than two some partial results are known. In this article we analyze a case that lies in the boundary between one and several complex variables. We consider skew product maps of the form $F(z,w)=(\la(z),f(z,w))$. We deal with the case of \textit{parabolic} skew product maps, that is when $DF(0,0)=\textrm{Id}$. Our goal is to describe the behavior of orbits around a whole neighborhood of the origin. We establish formulas for conjugacy maps in different regions of a neighborhood of the origin.
\end{abstract}

\maketitle

\section{Introduction}\label{section:intro}

The dynamics of skew-product maps $F(z,w) = (\la(z),f(z,w))$ has been studied in the past \cite{Jo, L, PR, PS, PV}. In this article we focus on a local aspect of the study, that is we look at the dynamics of $F$ close to a fixed point. By simplicity, suppose that the fixed point is the origin $(0,0)$. We turn our attention to a class of skew-product map that we call \textit{parabolic}, that is $\la(z)=z+O(|z|^2)$ and $f(z,w) = w +O(|(z,w)|^2)$. 

Skew-product maps are maps in which we can test general theorems for dynamics of self-maps on several dimensions. Since the first coordinate depends only on one coordinate, we can use the results from one complex dynamics to obtain information on one of the variables. Nonetheless, they provide a richer theory than in one dimension. An instance of this can be seen in the recent article by Astorg et al. \cite{ABDPR} in which they describe a skew-product map in two dimensions that has a wandering Fatou component.

We center our study on the following maps:
\begin{align}\label{Fintro}
F: &(\CC^2,0) \to (\CC^2,0)\\
\nonumber F(z,w) &= (\la(z),f(z,w)) 
\end{align}
where $\la(z) = z + a_2z^2 + O(z^3), a_2 \neq 0$ and $f(z,w) = w +b_2w^2 + O((z,w)^3)$, $b_2 \neq 0$.


Our goal is to describe the dynamics of a map above in a neighborhood of the origin. We divide our goals into the following two categories:
\begin{itemize}
\item[A.] Describe regions in which $F$ is conjugated to a simpler map. 
\item[B.] Find formulas for the conjugation map in each region, as in the one dimension case.
\end{itemize}

One classical tool in the study of local dynamics is a conjugacy map. Finding a conjugacy map to a simpler map depends strongly on the type of map we are studying and the dimension of our space. 

Consider $F: (\CC^n,p) \to (\CC^n,p)$ a holomorphic germ with a fixed point $p$. A local conjugacy of $F$ to $G$ is a one-to-one map $\phi: U_p \to \CC^n$ where $U_p \subset \CC^n$ is an open neighborhood around $p$ and where $G=\phi^{-1}\circ F\circ \phi$. In general the goal is to obtain a conjugacy to a map $G$, easier to study than $F$. There is a rich history that goes back to Schroeder about conjugacy. We point out the reader to \cite{Abate} or \cite{Milnor} for a very complete list of results. 

In the case of $F$ a parabolic map, that is $DF(p) = \textrm{Id}$, only some partial results are known. The dynamics of parabolic maps in several dimensions is in general very chaotic \cite{AT,Hak} and although some results have been proven for the case of generic maps, much less is known in general in comparison to the theory in one dimension.

One common feature on the study of parabolic maps is the conjugacy to a translation. While normally this conjugacy cannot be realized on a whole neigborhood around $p$, it is well defined on open sets that have $p$ at its boundary. The conjugacy is commonly referred as a Fatou coordinate. 

Fatou coordinates are useful tools on the study of parabolic maps. In one dimension for instance, it is the main tool in order to study parabolic bifurcation. In recent work of Bedford, Smillie and Ueda \cite{BSU}, they prove parabolic bifurcation results in the semi parabolic case by using Fatou coordinates. 

Let us recall the results in one dimension theory. Consider the map $f(z) = z + a_2z^2 + O(z^3), a_2 \neq 0$, where the origin is a parabolic fixed point for $f$. The Leau Fatou flower theorem states that there exists a parabolic basin $B$ for the origin, that is an open set with the origin at its boundary such that every point converges to the origin after iteration by $f$. There exists in fact a conjugacy of $f$ to the translation map $g(w)=w+1$ in the set $B$. Similarly there exists a repelling basin $R$ converging to $0$ under backward iteration. Likewise we can construct a conjugacy to the translation. Note that the union of $B$ and $R$ contains a full pointed neighborhood of the origin \cite{Milnor}. 

Our goal in this article is to describe the dynamics of a parabolic map in two dimensions in a similar way. That is, we would like to divide an entire neighborhood of the origin in several open sets, in such a way that we can conjugate our parabolic map to simpler maps in each one of those open sets.

Our main results are Theorems \ref{GgeneralWiWo} and \ref{GgeneralWaWb}. Let us summarize the content of the theorem here.

\begin{thm*}\label{MAINTHEOREM} Let $F$ as in \eqref{Fintro}. Then we can describe the dynamics of $F$ on a neighborhood of the $w$-axis. That is, after a change of coordinates for $F$, the following set:
$$
U = \{(z,w)\in\CC^2, |z|<\epsilon, |w|<\epsilon, |w|<|z|^M\} 
$$
where $M$ can be chosen as large as desired, can be divided on several regions such that in each region we have a conjugacy of $F$ to a simpler map.
\end{thm*}

While most of the conjugacy maps for the hyperbolic case can be obtained as a limit of iterates of our maps, Fatou coordinates are in general not so easily computed.  In this article we give formulas for Fatou coordinates for the class of skew-product parabolic maps as in \eqref{Fintro}. 

The article is organized as follows: In the following section we write down properties of Fatou coordinates, namely how do they change after changes of coordinates. In section 3 we recall results in one dimension. In section 4 we recall results from \cite{Vi14} where we found a complete description of the dynamics of a more particular class of parabolic maps on a whole neighborhood of the origin. Finally in the last section we prove the main theorem.

\medskip
\noindent

\section{Fatou coordinates}

Since we use Fatou coordinates of different maps throughout our article we write here the main definitions as well as properties. We work on the most general possible case. In the following sections we use the results from this section.

Let $F: (\CC^n,p) \to (\CC^n,p), n\geq 1,$ a holomorphic germ with a fixed point $p$ such that $DF(p) = \Id$ and $F\neq \Id$.
 Along our article we alternate between our fixed point $p$ being the origin and the point at infinity. The hypotheses on the derivative of $F$ guarantees that there is a local well defined inverse holomorphic germ $F^{-1}$ on a neighborhood of $p$. 				

Let $\zeta \in \CC^k$ then we write $T_{\zeta}: \CC^k \to \CC^k$ the translation map $T_{\zeta}(z)= z+\zeta$. Assume from now on $\zeta \neq 0$.

\begin{defn}\label{incoming} Let $\UiF \subset \CC^n$ an open set such that $p \in  \partial \UiF$ and $F(\UiF) \subset \UiF$. Then we say $\UiF$ is an \textit{attracting basin} of $F$. \\
Let $\UiF \subset \CC^n$ be an attracting basin of $F$. Assume we have a map $\phiiF: \UiF \to \CC^k$ a holomorphic map such that the following equation is satisfied:
\begin{align}
\phiiF(F(z))= \phiiF(z)+\zeta\qquad \textrm{or equivalently}\qquad\phiiF \circ F = T_\zeta \circ \phiiF:
\end{align}  
\begin{align*}
\begin{CD} 
\UiF @>F>> \UiF \\ 
@V\phiiF VV  @V\phiiF VV   \\ 
\CC^k @>T_{\zeta}>>  \CC^k
\end{CD}
\end{align*}

where $0\neq\zeta \in \CC^k$, then we say $\phiiF$ is an \textit{incoming Fatou map} for $F$ and $\UiF$ with translation $T_\zeta$.
\end{defn}
\begin{remark}\label{multiple}
If $\phiiF$ is an incoming Fatou map for $F$ and $\UiF$ with translation $T_\zeta$, then $\lambda\phiiF$ is an incoming Fatou map for $F$ and $\UiF$ with translation $T_{\lambda\zeta}$ for any $0\neq\lambda \in \CC$.
\end{remark}
Repelling basins as well as repelling Fatou maps are defined by considering the local inverse map $F^{-1}$.

\begin{defn}\label{outgoing}
Assume $\UoF \subset \CC^n$ an open set such that $p \in  \partial \UoF$ and $F^{-1}(\UoF) \subset \UoF$, that is $\UoF$ is an open attracting basin of $F^{-1}$. Then we say $\UoF$ is a \textit{repelling basin} of $F$.\\
Assume there exists an incoming Fatou map $\psi$ for $F^{-1}$ and $\UoF$ with translation $T_{-\zeta}$.
\begin{align*}
\begin{CD} 
\UoF @>F^{-1}>> \UoF \\ 
@V\psi VV  @V\psi VV   \\ 
\CC^k @>T_{-\zeta}>>  \CC^k
\end{CD}
\end{align*}

Under the assumption of $n=k$, assume we have a local inverse map for $\psi$, then we call this map $$\phioF:  \psi(\UoF) \subset \CC^k\to \UoF$$ the \textit{outgoing Fatou map} for $F$ and $\Uo$ with respect to $T_\zeta$.
Using the functional equation satisfied by $\psi$ and $F^{-1}$ we can easily see that:
\begin{align}
F(\phioF(z))= \phioF(z+\zeta)\qquad \textrm{or equivalently}\qquad F\circ\phioF = \phioF \circ T_\zeta.
\end{align}  
\end{defn}

We point out some remarks about the definitions above.
\begin{remark}
When $n=k=1$, we can assume without loss of generality that $\zeta=1$ and the incoming/outgoing change of coordinates is well-known as \textit{incoming/outgoing Fatou coordinates}. 
When $n=k\geq 2$ we could in principle simply use many copies of Fatou coordinates for $k=1$. We demand in fact one more condition, that our maps should be injective (necessary for the outgoing Fatou map), and we call them also \textit{incoming/outgoing Fatou coordinates}. 
\end{remark}
\begin{remark} When $n \geq 2$ and $k=1$, the incoming change of coordinate has been used in the past to prove the existence of Fatou-Bieberbach maps for automorphisms of $\CC^n$ \cite{Hak,Vi12}.
\end{remark}
\begin{remark}
It is easy to see that Fatou coordinates are not unique. From the functional equations we see that compositions (resp. pre-compositions) of translations with incoming (resp. outgoing) Fatou coordinates are also incoming  (resp. outgoing) Fatou coordinates. 	
\end{remark}
When there is no danger of confusion about the map $F$, we simply write $\phii$ and $\phio$. For now though, we stick with the superscript referring to the map in question, since we want to establish how do they change when changing coordinates for our map $F$.

\begin{prop}\label{inversemap} Let $F$ be a germ as above. Assume $F$ and $F^{-1}$ have attracting basins $\UiF$ ad $\UoF=U^{\ii,F^{-1}}$. Then $F^{-1}$ also has a repelling basin, namely $U^{\oo,F^{-1}}=\UiF$. 
Let also $\phiiF$ (resp. $\phioF$) an incoming (resp. outgoing) Fatou coordinates for $F$ and $\UiF$ with respect to $T_\zeta$. Then the following 
\begin{align}
\phi^{o,F^{-1}}(z) = (\phi^{\textrm{i},F})^{-1}(-z),\qquad\phi^{\textrm{i},F^{-1}}(z) = -(\phi^{o,F})^{-1}(z)
\end{align}
give us outgoing (resp. incoming) Fatou maps for $F^{-1}$ and $U^{\oo,F^{-1}}$ (resp. $U^{\ii,F^{-1}}$) with respect to $T_\zeta$.
\end{prop}
\begin{proof}
The proof follows by verifying the respective equations and using Remark \ref{multiple}.
\end{proof}
\begin{prop}\label{cofc}
Let $\eta$ be a (local) change of coordinates between $F$ and $G$ as in the following commutative diagram:
\begin{align}
\begin{CD} 
(\CC^n,p) @>F>> (\CC^n,p) \\ 
@A\eta AA  @A\eta AA   \\ 
(\CC^n,q) @>G>>  (\CC^n,q)
\end{CD}
\end{align}
Assume we have $\UiF$ and $\UoF$, attracting/repelling basins for $F$ along with $\phiiF$ and $\phioF$ Fatou coordinates for $F$, then we can also find attracting/repelling basins for $G$ as well as incoming/outgoing Fatou coordinates for $G$ as follows:
\begin{align}\label{coc}
\UiG &= \eta^{-1}(\UiF),\qquad\UoG = \eta^{-1}(\UoF),\\
\nonumber\phiiG &= \phiiF \circ \eta, \qquad\phioG = \eta^{-1} \circ \phioF. 
\end{align}
where $\phiiG$ (resp. $\phioG$) is defined in $\UiG$ (resp. $\UoG$).
\end{prop}
The proof is immediate. 

\begin{remark} One observation that we will use repeatedly on the next sections is the following: on the last proposition we do not need $\eta$ to be well defined on a whole neighborhood of the origin. In fact $\eta$ can be defined only on $\UiG$ or similarly only on $\UoG$. 
\end{remark}




\section{Fatou coordinates in one dimension}

Consider a germ at the origin of the following form:
$$
f(z)=z+ az^2 + O(|z|^3),
$$
where $a \neq 0$ a parabolic germ at the origin. By a simple change of coordinates we can always assume $a=-1$. The following is the classic theorem of Leau and Fatou. See \cite{Milnor} for details.

\begin{thm}(Leau-Fatou Theorem) Assume $f$ as above. Then there exists $\Uif$ and $\Uof$ for $f$, such that $\Uif \cup \Uof$ form a punctured neighborhood of the origin, as well as incoming and outgoing Fatou coordinates $\phiif:\Uif\to\CC$ as well as $\phiof:\psi(\Uof)\to\Uof$.  
\end{thm}

Before we continue, we write down an explicit choice for the sets $\Uif$ and $\Uof$:
Let
$$
\Ve = \{ \zeta \in \CC, |\zeta|<\epsilon, |\textrm{Arg}(\zeta)| < 3\pi/4\}
$$
then $\Uif = \Ve$ for $\epsilon$ small enough and similarly we can see that $\Uof= -\Ve$. 

\begin{figure}[h]
\def\svgwidth{4in}{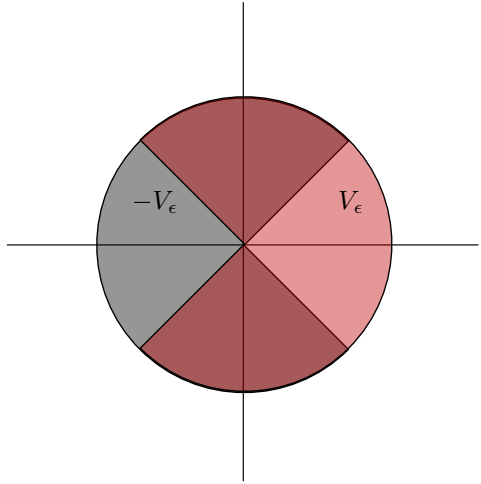}
\caption{attracting and repelling regions for $f$}
\label{fig:onedimension}
\end{figure}

We translate all the action to a neighborhood of $\infty$, using the inverse map $I(z) = 1/z$. We obtain $g(w) = w + 1 +\frac{\alpha}{w} +O(1/w^2)$ where $g=I \circ f\circ I$. Our fixed point is taken to be the infinity point. Using proposition \ref{cofc} we see that $\Uig$ and $\Uog$ can be obtained as the map $I$ applied to $\Uif$ and $\Uof$ respectively.
 Then $\phiig$ and $\phiog$ can be obtained as follows.
\begin{prop} Let $g(w) = w + 1 +\frac{\alpha}{w} +O\left(\frac{1}{w^2}\right)$, then we can find the incoming Fatou coordinate $\phiig: \Uig \to \CC$
of $g$ as the following limit:
\begin{align}\label{onein}
\phiig(w)&:= \lim_{n\to\infty}L_{-\alpha}(g^n(w))-n
\end{align}
Similarly we can find an extension to all of $\CC$ of the outgoing Fatou coordinate $\phiig: \CC \to \CC$ by using the following limit:
\begin{align}\label{oneout} 
\phiog(w)&:=\lim_{n\to\infty} g^n(L_\alpha(w-n))
\end{align} 
where $L_{\alpha}(w) = w + \alpha \log(w)$. 
\end{prop}

\begin{proof}
We start by proving equation \eqref{onein}. Note first that $\Uig = S_R=\{|w|>R, |\textrm{Arg}|<3\pi/4\}$, where $R=1/\epsilon$. 
In this set the map $L_{\alpha}$ is well defined an injective for any $\alpha \in \CC$. 
We use proposition 2.4 and the following lemma.

\begin{lemma} Consider $g$ as above, let $\rho=L_{-\alpha}\circ g \circ (L_{-\alpha})^{-1}$ defined on $W=L_{-\alpha}(\Uig)$ then $\rho(W)\subset W$ and $\rho(w)=w+1+O(\frac{1}{w^{1+\epsilon}})$. 
Similarly, consider $\tau=(L_{\alpha})^{-1}\circ g \circ L_{\alpha}$ defined on $V=(L_{\alpha})^{-1}(\Uog)$ then $\tau(V) \supset V$ and $\tau(w)=w+1+O(\frac{1}{w^{1+\epsilon}})$. 
\end{lemma}
\begin{proof}
Note that $L_\alpha$ and $L_{-\alpha}$ are one-to-one maps on $\Uig$ and $\Uog$. The rest of the assertions are immediate.
\end{proof}
Then $\phi^{\textrm{i},\rho}(w):= \lim_{n\to\infty}\rho^n(w)-n$ and similarly $\phi^{o,\tau}(w):= \lim_{n\to\infty}\tau^n(w-n)$
clearly converge. Using the fact that $\phiig=\phi^{\textrm{i},\rho}\circ L_{-\alpha}$ and $\phiog=L_\alpha \circ \phi^{o,\tau}$ we see that the limits on the statement of the proposition converge. 
\end{proof}




\section{Fatou coordinates in two dimensions}

Let us recall our results for a skew parabolic map $F$ of a particular form considered in \cite{Vi14}:
\begin{align}\label{Fspecial}
F(z,w) = \left(\frac{z}{1+z},f_z(w)\right) = \left(\frac{z}{1+z},w-w^2+w^3+O(w^4,z^4)\right)
\end{align}

Let $\Ve = \{ \zeta \in \CC, |\zeta|<\epsilon, |\textrm{Arg}(\zeta)| < 3\pi/4\}$. In one dimension, the union $\Ve \cup (-\Ve)$ forms a punctured neighborhood of the origin. In two dimension, to cover a full neighborhood of the origin we define the following four sets:
\begin{align}\label{regionsinU}
\Ui = \Ve \times \Ve, \Uo = (-\Ve)\times(-\Ve), \Ua= (-\Ve)\times \Ve, \Ub = \Ve\times(-\Ve), 
\end{align}
Then $\Ui \cup \Uo\cup\Ua\cup\Ub$ together with both axes, form a neighborhood of the origin. Since $F$ preserves the axis, then we can describe the dynamics in each axes by using the one dimensional results on parabolic maps.

As in the one dimensional case, we change variables so the fixed point is at infinity by using the map $I(z,w)=(1/z,1/w)$. Let $G = I\circ F\circ I$.
\begin{align}\label{Gspecial}
G(u,v) = \left(u+1,g_u(v)\right) = \left(u+1,v+1+O\left(\frac{1}{v^2},\frac{1}{uv^2}\right)\right)
\end{align}
Let $S_R=\{|\zeta|>R, |\textrm{Arg}(\zeta)|<3\pi/4\}$ where $R$ is large, since $I(\Ve)=S_R$. Then we will focus our study on the following four sets:
\begin{align}\label{regionsinW}
\Wi = S_R \times S_R, \Wo = -S_R\times-S_R, \Wa= -S_R\times S_R, \Wb =  S_R\times-S_R.
\end{align}

Denote $T_{(a,b)}: \CC^2 \to \CC^2$ defined as $T_{(a,b)}(z,w)=(z+a,w+b)$.

\begin{thm}
Let $G$ be as in \eqref{Gspecial}. Then in each region we have that the following limits exist:
\begin{itemize}
\item[(a)] For any $p \in \Wi$, then $G^n(p)$ converges to infinity. We have that $\PhiiG$ the Fatou coordinate is given by $\PhiiG:= \lim_{n\to\infty}T_{(-n,-n)} \circ  G^n$ and we have the following diagram:
\begin{align*}
\begin{CD} 
\Wi @>G>> \Wi \\ 
@V\PhiiG VV  @V \PhiiG VV   \\ 
\CC^2 @>T_{(1,1)}>>  \CC^2
\end{CD}
\end{align*}
\item[(b)] For any $p \in \Wo$, then $G^{-n}(p)$ converges to infinity. We have that $\PhioG$ the Fatou coordinate is given by $\PhioG:= \lim_{n\to\infty}G^n\circ T_{(-n,-n)}$ and we have the following diagram:
\begin{align*}
\begin{CD} 
G^{-1}(\Wo) @>G>> \Wo \\ 
@A\PhioG AA  @A \PhioG AA   \\ 
N \subset\CC^2 @>T_{(1,1)}>>  T_{(1,1)}(N) \subset \CC^2
\end{CD}
\end{align*}
where $N=(\PhioG)^{-1}(G^{-1}(\Wo))$ and $T_{(1,1)}(N) = (\PhioG)^{-1}(\Wo)$.
\end{itemize}
\end{thm}
\begin{proof}
We prove first the convergence of the sequence going to $\PhiiG$. Then we use this result to prove the analogue for $\PhioG$. 
Denote $\PhiiG_n:= T_{(-n,-n)} \circ G^n$.
A simple computation shows:
$$
\PhiiG_{n+1} - \PhiiG_n = G^{n+1}(u_n,v_n)-G^n(u_n,v_n)-(1,1) = (0, O(1/v_n^2,1/(u_nv_n^2))).
$$
since $u_n$ and $v_n$ are $O(n)$ when $(u,v)\in \Wi$ then $\PhiiG_n$ converge uniformly in compacts.\\
For the outgoing coordinate, we denote $\PhioG_n:=G^n\circ  T_{(-n,-n)}$. Then we have the following:
\begin{align}\label{phiGH}
\PhioG_n \circ \eta \circ \PhiiH_n \circ \eta = \Id.
\end{align}
for $H=\eta\circ G^{-1} \circ \eta$ and $\eta(u,v) = (-u,-v)$. Since $H(u,v) = (u+1,v+1+O(1/v^2))$, we have that $\PhiiH_n$ converges, and therefore $\PhioG_n$ also. Since $\eta(\Wi)=\Wo$ and $H(\Wi) \subset \Wi$ then $\Wo\subset\PhioG(\Wo)$.
\end{proof}

\begin{figure}[h]
\begin{center}
\def\svgwidth{3.5in}{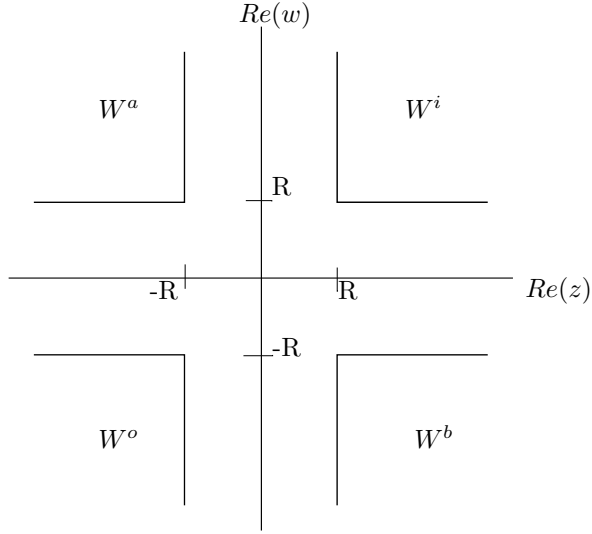}
\caption{Schematic figure of regions defined in \eqref{regionsinW}}
\label{fig:2dimensions}
\end{center}
\end{figure}

We also have:
\begin{thm}\label{GspecialWaWb}
Let $G$ be as in \eqref{Gspecial}. Then in each region we have that the following limits exist:
\begin{itemize}
\item[(a)] The following maps converge uniformly in compacts in $\Wa$ given by $\PsiaG:= \lim_{n\to\infty}T_{(n,-n)} \circ  G^n\circ T_{(-2n,0)}$ and we have the following diagram:
\begin{align*}
\begin{CD} 
\Wa @>(-1,g_\infty)>> \Wa \\ 
@V\PsiaG VV  @V \PsiaG VV   \\ 
\CC^2 @>T_{(-1,1)}>>  \CC^2
\end{CD}
\end{align*}
\item[(b)] The following maps converge uniformly in compacts in $\Wb$ given by $\PsibG:= \lim_{n\to\infty}T_{(-2n,0)} \circ  G^n\circ T_{(n,-n)}$ and we have the following diagram:
\begin{align*}
\begin{CD} 
L^{-1}(\Wb) @>L=(-1,g_\infty)>> \Wb \\ 
@A\PsibG AA  @A \PsibG AA   \\ 
E \subset \CC^2 @>T_{(-1,1)}>> T_{(-1,1)}(E) \subset \CC^2
\end{CD}
\end{align*}
\end{itemize}
\end{thm}

\begin{proof} 
Denote $\PsiaG_{n}:=T_{(n,-n)} \circ G^n \circ T_{(-2n,0)}$ and $\PsibG_{n}:=T_{(-2n,0)} \circ G^n \circ T_{(n,-n)}$. Unraveling we have:
$$
\PsibG_{n}(u,v)=(u,g_{u+2n-1}\circ\ldots \circ g_{u+n+1}\circ g_{u+n}(v-n))
$$
We have proven on \cite{Vi14} that the following map:
$$\psio_n(v) = g_{-v+\alpha+2n-1}\circ\ldots \circ g_{-v+\alpha+n+1}\circ g_{-v+\alpha+n}(v-n)
$$
converges for any $\alpha \in \CC$ and $v \in -S_R$ and that the limit $\psio(v+1) = g_{\infty}(\psio(v))$. 
Applying this result for $\alpha = u-v$ we obtain the convergence of the sequence $\PhiaG_n$ to the following map:
$$
\PsibG_n(u,v) \to (u, \psio(v)) 
$$
where $\psio(v+1) = g_{\infty}(\psio(v))$.

For the other  coordinate, we use the following identity:
\begin{align}\label{psiGH}
\PsiaG_n \circ \eta \circ \PsibH_n \circ \eta = \Id.
\end{align}
for $H=\eta\circ G^{-1} \circ \eta$, where $\eta(u,v) = (-u,-v)$. Since $H(u,v) = (u+1,v+1+O_u(1/v^2))$, we have that $\PsibH_n$ converges, and therefore $\PsiaG_n$ also. 
\end{proof}

Summarizing the results that we obtained for the map $F$:

\begin{thm}
Let $F$ be as in \eqref{Fspecial}. Let the sets $\Ui, \Uo, \Ua$ and  $\Ub$ defined as in \eqref{regionsinU}. Then:
\begin{itemize}
\item[(a)] The union of $\Ui, \Uo, \Ua$ and $\Ub$ together with the axes form a neighborhood of the origin in $\CC^2$.
\item[(b)] For any $p \in \Ui$ then $F^{n}(p) \in\Ui$ and $F^n$ converges to the origin uniformly in compacts. 
\item[(c)] For any $p \in \Uo$ then $F^{-n}(p) \in\Uo$ and $F^{-n}$ converges to the origin uniformly in compacts. 
\item[(d)] For any $p \in \Ua$ then $F^{-n}(p)$ converges to the $w$-axis, the invariant fiber of the map $F$. 
\item[(e)] For any $p \in \Ub$ then $F^n(p)$ converges to the $w$-axis, the invariant fiber of the map $F$.
\end{itemize}
\end{thm}


\section{General Case}

We are ready now to tackle the most general case. Consider now the following map:
\begin{align}
F(z,w) = (\lambda(z),f_z(w))
\end{align}
where $\lambda(z) = z + O(z^2), f_z(w) = w+O(|(z,w)|^2)$.

We focus on the following case:
\begin{align*}
F(z,w) = (z+a_2z^2+O(z^3),w+b_{2}w^2+O(|(z,w)|^3))
\end{align*}
where $a_2 \neq 0$ and $b_2 \neq 0$. 

By a simple change of coordinates we can assume $a_2 = -1$ and $b_2=-1$. Using a shear polynomial change of coordinates we can increase the power of $z$ on the second term. Similarly by using another change of coordinates we can increase the degree of the $z$ term that is multiplied by $w$. Therefore we can write $F$ as follows:
\begin{align}\label{genF}
F(z,w) = (z-z^2+O(z^3),w-w^2+O(w^3,zw^2,z^Mw,z^M))
\end{align}
for $M$ as large as we want.

We once again study the map at infinity by using the inverse map $(u,v):=I(z,w)=(1/z,1/w)$, so we consider $G=I \circ F \circ I$ and the fixed point is at infinity.
\begin{align}\label{genG}
G(u,v) =  (\rho(u),g_u(v))= \left(u+1+O\left(\frac{1}{u}\right),v+1+ O\left(\frac{1}{u},\frac{1}{v},\frac{v}{u^M},\frac{v^2}{u^M}\right)\right)
\end{align}

Let $S_R = \{|\zeta|>R, |\textrm{Arg}(\zeta)|<3\pi/4\}$ for $R$ large. Since the first coordinate of $G$ depends only on one variable, there exists two maps $\psi_1$ and $\psi_2$, defined respectively on $S_R$ and $-S_R$ such that $\psi_j^{-1} \circ \rho \circ \psi_j(u) = u+1$ for $u \in S_R$ when $j=1$, and for $u \in -S_R$ when $j=2$.

Recall that the maps $\psi_j$ are of the form $\psi_j(u) = u +O(\log(u))$, for $j=1,2$. We conjugate $G$ on $T_1 = S_R \times (S_R\cup (-S_R))$ by the change of coordinates $\Psi_1(u,v) = (\psi_1(u),v)$ and on the set $T_2 = (-S_R) \times (S_R\cup (-S_R))$ by the map  $\Psi_2(u,v) = (\psi_2(u),v)$.

Then we obtain the following conjugations of $G$:
\begin{align}
G_{j}(u,v) =  \left(u+1,v+1+ O\left(\frac{1}{u},\frac{\log(u)}{u^2},\frac{1}{v},\frac{v}{u^M},\frac{v\log(u)}{u^{M+1}},\frac{v^2}{u^M},\frac{v^2\log(u)}{u^{M+1}}\right)\right)
\end{align}
where each $G_j = (\Psi_j)^{-1}\circ G \circ \Psi_j$ is defined in $T_j$, for $j=1,2$. The specific higher order terms are in general different on each region.

Now, as before we separate each $T_i$ in two sets:
\begin{align*}
T_1 &= S_R \times (S_R\cup (-S_R))  = (S_R \times S_R) \cup (S_R \times (-S_R)) = \Wi \cup \Wb,\\
T_2 &= (-S_R) \times (S_R\cup (-S_R))  = (-S_R \times S_R) \cup (-S_R \times -S_R) = \Wa \cup \Wo.
\end{align*}

From now on, when we refer to a region $W$, we mean one of the possible four sets $\Wi,\Wb,\Wa$ or $\Wo$.

Consider the class of maps $\Theta(u,v) = (u,v+\alpha \log(u) + \beta \log(v))$. It is immediate to see that after choosing $R$ large enough then $\Theta$ is an injective transformation in each region $W$.

Now we can see that conjugating $G_j$ by $\Theta$ and choosing $\alpha$ and $\beta$ appropriately, we can get rid of the linear terms $O(1/u,1/v)$ and we obtain:
$$
\Theta^{-1} \circ G_j \circ \Theta(u,v) = \left(u+1,v+1+O\left(\frac{\log(u)}{u^2},\frac{1}{v^2},\frac{\log(u)}{v^2},\frac{\log(u)\log(v)}{v^3},\frac{v}{u^M},\frac{v^2}{u^M}\right)\right).
$$

To emphasize that each one of these maps is a different conjugation of $G$ on each set $W$, we write $\thetai$ the change of coordinates on $\Wi$, $\thetao$ on $\Wo$, $\thetaa$ on $\Wa$ and $\thetab$ on $\Wb$. We write $\Gi:= (\thetai)^{-1}\circ G_1 \circ \thetai$ the corresponding map defined on $\Wi$, $\Go:=(\thetao)^{-1}\circ G_2 \circ \thetao$ on $\Wo$, $\Ga:=(\thetaa)^{-1}\circ G_2 \circ \thetaa$ on $\Wa$ and $\Gb:=(\thetab)^{-1}\circ G_1 \circ \thetab$ on $\Wb$.

We have that the following subsets of $\Wi$ and $\Wo$ be the respective attracting and repelling basins for $\Gi$ and $\Go$:
\begin{align}
\widetilde{\Wi} &= \{(u,v) \in S_R\times S_R, |u|^{M+1}>|v|\}\\
\nonumber\widetilde{\Wo} &= \{(u,v)\in (-S_R)\times(-S_R),|u|^{M+1}>|v|\}
\end{align}

Let $\WiG=\Psi_1\circ \theta_i (\widetilde{\Wi})$, then we have that $G(\WiG) \subset \WiG$. Using results from the last section, we also have that we can conjugate $G_{\textrm{i}}$ to a translation on $\widetilde{\Wi}$ by using the limit:
\begin{align*}
\Phi^{\textrm{i},G_\textrm{i}}: \widetilde{\Wi} \to \CC^2, \quad \Phi^{\textrm{i},G_\textrm{i}}= \lim_{n\to\infty}T_{(-n,-n)}\circ G_\textrm{i}^n.
\end{align*}
If we unravel for $G$ then we obtain the following formula for the Fatou coordinate on the incoming basin for $G$:
\begin{align}
\PhiiG: \WiG \to \CC^2, \quad \PhiiG(u,w)= \lim_{n\to\infty}T_{(-n,-n)}\circ \thetai^{-1}\circ\Psi_1^{-1}\circ G^n
\end{align}

Similarly, if we define $\WoG=\Psi_2\circ \theta_o (\widetilde{\Wo})$ then we obtain that  $G(\WoG) \supset \WoG$. By the theorem on the last section we obtain a conjugation of $G$ on $\WoG$ to the translation:
\begin{align*}
\Phi^{\textrm{o},G_\textrm{o}}: \widetilde{\Wo} \to \CC^2, \quad \Phi^{\textrm{o},G_\textrm{o}}= \lim_{n\to\infty}G_\textrm{o}^n\circ T_{(-n,-n)}.
\end{align*}
Unraveling for $G$ then we obtain the following formula for the Fatou coordinate on the outgoing basin for $G$:
\begin{align}
\PhioG: \WoG \to \CC^2, \quad \PhioG(u,w)= \lim_{n\to\infty} G^n \circ \Psi_2 \circ \thetao \circ T_{(-n,-n)}.
\end{align}

We have therefore proven:
\begin{thm}\label{GgeneralWiWo} Let $G$ as above on \eqref{genG}. Then we can find incoming and outgoing Fatou coordinates for the respective incoming and outgoing basins for $G$ at infinity. 
\end{thm}

We also obtain information on the behavior of $G$ on the regions $\Wa$ and $\Wb$, since we can apply theorem \ref{GspecialWaWb} to the maps $\Ga$ and $\Gb$ respectively.  Even though we do not have a conjugacy of $\Ga$ or $\Gb$ on the regions $\Wa$ and $\Wb$, we have that certain compositions of $G$ on these regions (together with translation maps) conjugate to the action $(-1,g_\infty)$. 

On theorem \ref{GspecialWaWb} we proved that on the region $\Wa$, the map given by $\Psia:= \lim_{n\to\infty}T_{(n,-n)} \circ  \Ga ^n\circ T_{(-2n,0)}$ satisfies $\Psia \circ (-1,g_\infty) = T_{(-1,1)}\circ \Psia$. Since $\Ga = (\thetaa)^{-1}\circ G_2 \circ \thetaa$, then $\Ga^n = (\thetaa)^{-1}\circ (\Psi_2)^{-1} \circ G^n \circ \Psi_2 \circ \thetaa$, and  $\Psia= \lim_{n\to\infty}T_{(n,-n)} \circ   (\thetaa)^{-1}\circ (\Psi_2)^{-1} \circ G^n \circ \Psi_2 \circ \thetaa\circ T_{(-2n,0)}$ converges and satisfy the same commutative diagram.

Similarly, on the region $\Wb$, the map given by $\Psib:= \lim_{n\to\infty}T_{(-2n,0)} \circ  \Gb ^n\circ T_{(n,-n)}$ satisfies $\Psib \circ (-1,g_\infty) = T_{(-1,1)}\circ \Psib$. Since $\Gb = (\thetab)^{-1}\circ G_1 \circ \thetab$, then $\Gb^n = (\thetab)^{-1}\circ (\Psi_1)^{-1} \circ G^n \circ \Psi_1 \circ \thetab$, then  $\Psib= \lim_{n\to\infty}T_{(-2n,0)} \circ (\thetab)^{-1}\circ (\Psi_1)^{-1} \circ G^n \circ \Psi_1 \circ \thetab\circ T_{(n,-n)}$ converges and satisfy the same commutative diagram.

We have therefore proven 
\begin{thm}\label{GgeneralWaWb} Let $G$ as above on \eqref{genG}. Then on the regions: 
\begin{align}
\widetilde{\Wa} &= \{(u,v) \in -S_R\times S_R, |u|^{M+1}>|v|\}\\
\nonumber\widetilde{\Wb} &= \{(u,v)\in S_R\times(-S_R),|u|^{M+1}>|v|\},
\end{align}
the following limits exist: $\Psia= \lim_{n\to\infty}T_{(n,-n)} \circ   (\thetaa)^{-1}\circ (\Psi_2)^{-1} \circ G^n \circ \Psi_2 \circ \thetaa\circ T_{(-2n,0)}$ and $\Psib=\lim_{n\to\infty}T_{(-2n,0)} \circ (\thetab)^{-1}\circ (\Psi_1)^{-1} \circ G^n \circ \Psi_1 \circ \thetab\circ T_{(n,-n)}$ and the second coordinate conjugates $g_\infty$ to the translation $T_{1}$.
\end{thm}


\end{document}